\documentclass[11pt]{article}
\usepackage{epsf}
\usepackage{amsbsy,amsmath}
\usepackage{mathtools}
\usepackage{mathrsfs}
\usepackage{amsfonts}
\usepackage{amssymb}
\usepackage{enumitem}
\usepackage{eucal}
\usepackage{enumitem}
\usepackage{graphics,mathrsfs}
\usepackage{amsthm}
\usepackage{secdot}
\usepackage{xcolor}
\newtheorem{theorem}{Theorem}[section]
\newtheorem{lemma}[theorem]{Lemma}

\theoremstyle{definition}

\theoremstyle{remark}
\theoremstyle{claim}
\newtheorem{remark}[theorem]{Remark}
\numberwithin{equation}{section}

\numberwithin{equation}{section}

\newsavebox{\savepar}

\begin{document}
	
	\title{\bf	Analysis of solution to an elliptic free boundary value problem equipped with a `bad' data}

	\author{ Debajyoti Choudhuri$^{a,}$\footnote{E-mail: dchoudhuri@iitbbs.ac.in, ORCID ID: 0000-0001-8744-9350},\, Shengda Zeng$^{b,}$\footnote{Corresponding author. E-mail: zengshengda@163.com}\\
		\small{$^{a}$ School of Basic Sciences, Indian Institute of Technology Bhubaneswar,}\\
		\small{  Khordha - 752050, Odisha, India.	}
		\\
		\small{$^{b}$ Center for Applied Mathematics of Guangxi,}\\
		\small{and Guangxi Colleges and Universities Key Laboratory of Complex System Optimization,}\\
		\small{ and Big Data Processing, Yulin Normal University, Yulin 537000, Guangxi, P.R. China.}}		
	
	\date{\today}
	
	\maketitle
	
	\begin{abstract}
		We will study a free boundary value problem driven by a source term which is quite {\it irregular}. In the process, we will establish a monotonicity result, and regularity of the solution.
		\begin{flushleft}
		\end{flushleft}
	\end{abstract}
	\section{Introduction}
	Free boundary value problems (FBVPs) are those where the PDE has boundary condition on an unknown interface/boundary. Besides the existence of solution, a very interesting question to study is the regularity of solution(s) if it exists. A celebrated work in this direction is the work due to Caffarelli \cite{LC1}. Besides mathematical interest, the problem also appears in Physics, Geometry etc. Inspired from \cite{LC1}, we will study the following problem:
	\begin{equation}\label{main}
		\left\{ \begin{array}{ll}-\Delta u=f\chi_{\{u>0\}}~~~\text{in}~\Omega\\
			u=0~~~\text{in}~\partial\Omega,\end{array} 
		\right.
	\end{equation}
	where $f\in L^1(\Omega)$ or a Radon measure, $\Omega\subset\mathbb{R}^N$ is a bounded domain of sufficiently smooth boundary. However when $f\in L^1(\Omega)$, it is usually required to probe the existence of an {\it Entropy solution} instead of a weak solution since the boundedness of the derivatives becomes a concern. A very natural question to this PDE besides the existence and multiplicty is the regularity of solution(s). Here I would like to bring to the notice of the readers the work of Perera \cite{KP1} which provides a very systematic approach to establish the regularity of the solutions to the free boundary value problem considered over there. The author in \cite{KP1} also considered the presence of a nonlinear term. On similar lines, Choudhuri-Repov\v{s} \cite{DC-DDR1} established the existence of solution to a {\it Prandtl-Batchelor} type problem with free boundary conditions. The authors in \cite{DC-DDR1} also derived a monotonicity condition which is necessary in establishing the regularity of solution to the considered problem. A noteworthy book that has well documented the theory of FBVPs is due to Petrosyan et al. \cite{AP-HS-NU1}.
	
	The consideration of the source term being irregular is new. Note that this can be treated as Radon measure as well, however to my knowledge, the consideration of the problem \eqref{main} with a purely Radon measure is still open. At this juncture, it is of interest to ask that how will the regularity of the solution to \eqref{main} change with such an {\it irregular} source term?. The manuscript henceforth, will revolve around answering this question. Intuitively, the regularity of the solution to \eqref{main}, say $u$, influences the regularity of the free boundary $\partial\{u>0\}$. 
	
	A quick sneak-peek into our findings are as follows:  
	The first lemma that we prove will be to establish the order of growth of the solution to \eqref{main}.
		\begin{lemma}({\bf growth of $u$})\label{quad_growth}
		Let $u$ be any solution in $B_{1/2}$, and let $x_0\in \bar{B}_{1/2}$ be any point on $\{u=0\}$. Then for any $r\in\left(0,\frac{1}{4}\right)$ we have $$0\leq \underset{x\in B_{r}(x_0)}\sup u(x)\leq C_rr^{2-\frac{N}{q}}.$$
	\end{lemma}
	The next theorem is to establish the regularity measure of the solution.
		\begin{theorem}\label{optimal_reg}
		Suppose $f\in L^q(B_1)$ and $u$ be a solution to \eqref{main}. Then $u\in C(B_{1/2})$ satisfying the estimate 
		\begin{align}\label{ineq2}
			\begin{split}
				\|u\|_{C_{\text{loc}}^1(B_{1/2})}\leq &C_r(\|u\|_{\infty}+\|f\|_{L^q(B_1)}).
			\end{split}
		\end{align}
	\end{theorem}
	The weak degeneracy of the solution finally establishes that the solution is of the order of $r^{2-\frac{N}{q}}$ for any opints close to the boundary of $\{u>0\}$.	
		\begin{lemma}({\bf weak nondegeneracy})\label{nondegeneracy}
		Let $f\geq c_0>0$ in $B_1$. Then on the boundary $\partial \{u>0\}\cap B_{1/2}$, we have $$\underset{\partial B_r(x_1)}\sup u\geq Cr^{2-\frac{N}{q}},$$
		where $u$ is the solution to \eqref{main}.
	\end{lemma}
	In the process of proving the above results we required the service of a {\it Weiss-like identity} which has been proved in the Appendix.
	\subsection{Physical relevance of the problem}
	Some of the classic problems that depicted the free boundary nature are the Stefan problem, the Hele-Shaw flow, the {\it Dam} problem etc. However, a simple physical phenomena, namely the non-equilibrium system of melting of ice is a very relatable example. In a given block of ice, the heat equation can be solved with a given set of appropriate initial/boundary conditions in order to determine the temperature. However, if there is a region of ice in which the temperature is greater than the melting point of ice, this subdomain will be filled with water. The boundary thus formed due to the ice-water interface is controlled by the solution of the heat equation. The free boundary thus corresponds to the interface between water and ice. Therefore a {\it free boundary} in the nature is not unnatural. The problem in this paper is a fair generalization to this physical phenomena which besides being a new addition to the literature can also serve as a note to find some important results pertaining to the FBVP.
	
	\section{Main results}
	Let us first define the associated energy functional to the problem \eqref{main} as follows:
	$$I(u):=\frac{1}{2}\int_{\Omega}|\nabla u|^2dx-\int_{\Omega}fu^+dx.$$
	
	The first result is stated by the following lemma which establishes the unique solvabity of \eqref{main}.
	\begin{lemma}\label{exist_unique}
		Suppose $\Omega\subset\mathbb{R}^N$ is a bounded domain with Lipschitz boundary, and $g:\partial\Omega\to\mathbb{R}$ such that 
		$$S:=\{u\in H^1(\Omega):u\geq 0~\text{in}~\Omega, u|_{\partial\Omega}=g\}\neq\varnothing.$$
		Then for any $f\in L^2(\Omega)$ there exists a unique minimizer of $I$ in $H^1(\Omega)^+$ and $u|_{\partial\Omega}=g$.
	\end{lemma}
	\begin{proof}
		We begin by observing that for any $u\in (H^1(\Omega))^+$ we have the following:
		\begin{eqnarray}\label{ineq1}
			\begin{split}
				I(u)=\frac{1}{2}\|u\|^2-\int_{\Omega}fudx
				\geq  \frac{1}{2}\|u\|^2-c_1\|f\|_2\|u\|.
			\end{split}
		\end{eqnarray}
		Hence this implies that $I(u)\to\infty$ as $\|u\|\to\infty$. In addition to this we also have $I(0)=0$. From the above observations and $I\in C^1(H^1(\Omega)^+)$. 
		
		Fix a nonzero $u\in H^1(\Omega)^+$ and consider the fiber map 
		$$I(tu)=\frac{t^2}{2}\|u\|^2-t\int_{\Omega}fudx.$$
		
		We collect the point(s) at which $\frac{d}{dt}I(tu)$ vanishes. We see that $t^*=\frac{\int_{\Omega}fudx}{\|u\|^2}$ is the only such point at which $\frac{d^2}{dt^2}I(t^*u)>0$. Hence $t^*u$ is a minimizer. However it is easy to see that there exists a unique minimizer. This is because
		\begin{align}
			\begin{split}
				0=&	(\langle \nabla u,\nabla \phi\rangle-\int_{\Omega}f\phi dx-\int_{\Omega}\phi d\mu)-	(\langle \nabla v,\nabla \phi\rangle-\int_{\Omega}f\phi dx)
				= \langle \nabla u-\nabla v,\nabla \phi\rangle
			\end{split}
		\end{align}
		for every $\phi\in H^1(\Omega)^+$. In particular we choose $\phi=(u-v)^+$ to obtain $(u-v)^+=0$. Similarly we obtain $(u-v)^-=0$, and hence $u=v$. 
	\end{proof}
	
	\begin{remark} 
		Henceforth, a ball  centred at $z$ and of radius $r$ units will be denoted by $B_r(z)$.
	\end{remark}
	The following result is about an optimal regularity of solution to \eqref{main}. We will show it for $q=1$. the remaininng cases of $q>1$ follows from it.

	\begin{proof}[Proof of Lemma \ref{quad_growth}]
		Let $x_0\in \bar{B}_{1/4}\cap \partial\{v>0\}$. By Harnack's inequality we have 
		$$\underset{B_{r/2}(X_0)}\sup u\leq C_r\left(\underset{B_{r/2}(x_0)}\inf u+r^{2-\frac{N}{q}}\|f\chi_{\{u>0\}}\|_{L^1(B_{r}(x_0))}\right).$$
		However, since $u\geq 0$ and $u(x_0)=0$, hence $\underset{B_{r/2}(X_0)}\sup u\leq C_rr^{2-\frac{N}{q}}\|f\chi_{\{u>0\}}\|_{L^1(B_{r}(x_0))}.$
		This agrees with the fact that $u$ grows rapidly on the free boundary when $f\in L^q(\{u>0\})$ with $q<N/2$. Therefore, the question of $u$  being H\"{o}lder continuous does not arise. When $q=N/2$, the growth of $u$ is inconclusive. Therefore when $q>N/2$, we see that the solution has a growth of the order of $r^{2-\frac{N}{q}}$, leading to a quadratic growth when $q$
		is chosen to be $\infty$.
	\end{proof}

	\begin{proof}[Proof of Theorem \ref{optimal_reg}]
		We scale $u$ in such a way that $\|u\|_{\infty}+\|f\|_{L^q(B_1)}<1$. The $r^{2-\frac{N}{q}}$ growth of $u$ yields H\"{o}lder continuous growth of its first derivative. Let $x_1\in \{u>0\}\cap B_{r/2}$ and $x_2\in\partial\{u>0\}$ be the closest boundary point to $x_1$. Further let us define $d:=|x_2-x_1|$. Apparently, we have that $-\Delta u=f$ in $B_d(x_2)$.
		
		By Schauder estimates which is proved by the Calder\'{o}n-Zygmund estimates, we obtain 
		$$\|u\|_{W^{2,q}(B_{d/2}(x_2))}\leq C\left(\|u\|_{L^{\infty}(B_d(x_2))}+\|f\|_{L^q(B_d(x_2))}\right).$$
		By Lemma \ref{quad_growth} we have that $\|u\|_{L^{\infty}(B_d(x_2))}\leq cd^{2-\frac{N}{q}}\leq C|B_1|$. Also since $N/2<q<N$, hence $\|u\|_{C^{1,\alpha}}(\bar{B}_{d/2}(x_2))\leq C\|u\|_{W^{2,q}(B_{d/2}(x_2))}$, where 
		\begin{align*}
			\begin{split}
				\alpha
				&=\begin{cases}
					1-\frac{N}{q}+\left[\frac{N}{q}\right],& ~\text{if}~\frac{N}{q}\notin\mathbb{Z}\\
					\text{any number}~<1,&~\text{if}~\frac{N}{q}\in\mathbb{Z}.
				\end{cases}
			\end{split}
		\end{align*}
		Hence $u\in C^{1,\alpha}(B_{1/2})$ since $x_2$ was arbitrarily chosen from $\{u>0\}\cap B_{1/2}$.
	\end{proof}
	\begin{remark}\label{loss of regularity}
		Lesser regularity of $f$ made us lose regularity of $u$. Had $f$ been bounded, $u$ would have had $C^{1,1}$ regularity.
	\end{remark}
	
	The above result in Lemma \ref{quad_growth} shows that the solution has a growth not exceeding $r^{2-\frac{N}{q}}$. Due to this it is enough to consider the following problem:
	\begin{align}\label{main_mod}
		\begin{split}
			-\Delta u= f\chi_{\{u>0\}}\mbox{ and }
			u\geq  0,~\text{in}~B_1,
		\end{split}
	\end{align}
	with $0$ being a free boundary point and $f>0$ in $\{u>0\}$. We will now show that at free boundary points $u$ grows at least as fast as $r^{2-\frac{N}{q}}$. 
	
	\begin{proof}[Proof of Lemma \ref{nondegeneracy}]
		Let $u$ be a solution to \eqref{main} and $f\geq c_0>0$ in $B_1$ and define $v(x):=u(x)-\frac{c_0}{2N}|x-x'|^{2-\frac{N}{q}},$	
		where $x'\in\{u>0\}$ closest to $x_1$. On testing the weak formulation of the problem $-\Delta v=-\Delta u-\frac{c_0}{2q'}\left(2-\frac{N}{q}\right)|x-x'|^{-N/q}$ with $\varphi~(\geq 0)\in W^{1,2}(\{u>0\})$ we obtain 
		\begin{align}\label{weak1}
			\begin{split}
				\int_{\{u>0\}}\nabla v\cdot\nabla\varphi dx=& \int_{\{u>0\}}\nabla u\cdot\nabla\varphi dx
				+\frac{c_0}{2q'}\left(2-\frac{N}{q}\right)\int_{\{u>0\}}|x-x'|^{-N/q}\varphi dx\\
				=& \int_{\{u>0\}}\left(f+\frac{c_0}{2q'}\left(2-\frac{N}{q}\right)\right)|x-x'|^{-N/q}\varphi\geq 0
			\end{split}
		\end{align}
		for each $\varphi~(\geq 0) \in W^{1,2}(\{u>0\})$. Thus $-\Delta v\geq 0$ in $\{u>0\}\cap B_r(x')$. Since $v(x')>0$, we have by the continuity of $u$ and the maximum principle that a maximum of $v$ (which is positive) is attained on the boundary $\partial(\{u>0\}\cap B_r(x'))$. However, on the free boundary $\partial\{u>0\}$ we apparently have $v<0$. Thus there must exist a point on $\partial B_r(x')$ at which $v>0$. Hence 
		$$0<\underset{\partial B_r(x')}\sup \left(u(x)-\frac{c_0}{2N}r^{2-\frac{N}{q}}\right).$$
		On passing the limit $x'\to x_1$ we get $\underset{\partial B_r(x_1)}\sup u\geq \frac{c_0}{2N}r^{2-\frac{N}{q}} .$
	\end{proof}

	Finally we will show that a blow up solution solves the problem \eqref{main}. 
	
	\begin{theorem}\label{main_result}
	Suppose $u$ is a solution to \eqref{main} and let $u_r(x):=\frac{u(rx)}{r^{2-\frac{N}{q}}}$. Then for any sequence $r_n\to 0$, there exists a subsequence (still denoted by $r_n$) such that $u_{r_n}\to u_*$ as $n\to\infty $ in $C_{\text{loc}}^{1}(\mathbb{R}^N)$. This $u_*$ obeys
	\begin{align}\label{modified_prob}
		\begin{split}
		-\Delta u_*= f~\text{in}~B_1,\,
		u_*\geq  ~\text{in}~B_1,\mbox{ and }
	 0~\text{is a }\text{ free boundary point}. 
		\end{split}
		\end{align}
		Furthermore, $u_*\in C_{\text{loc}}^{1,\alpha}(\mathbb{R}^N)$, $f(>0)\in L^q(\Omega)$.
	\end{theorem}
	\begin{proof}
	By the nondegenracy result proved in \ref{nondegeneracy} we have $$C^{-1}\leq\underset{B_1}\sup~u_r\leq C.$$
	Since $u$ is $C^{1,2-\frac{N}{q}}$ regular, we have $\|Du_r\|_{L^{\infty}(B_{1/r})}\leq C.$
	Now since the sequence $(u_{r_n})$ is uniformly bounded in $C^{1,2-\frac{N}{q}}(\mathcal{C})$, for $r_n\to 0$, for each compact subset $\mathcal{C}$ of $\mathbb{R}^N$, hence by the {\it Ascoli-Arzela} theorem there exists a subsequence, still denoted by $(r_n)$, yields $u_{r_n}\to u_*~\text{in}~C_{\text{loc}}^1(\mathbb{R}^N),$
	with $u_*\in ^{1,2-\frac{N}{q}}(\mathcal{C})$.  Also $\|Du_*\|_{L^{\infty}(\mathcal{C})}\leq C$.
	
		We now show that $\Delta u_*=1$ in $\{u_*>0\}\cap\mathcal{C}$. Let $\varphi\in C_c^{\infty}(\{u_*>0\}\cap\mathcal{C})$ we have
		\begin{align}\label{solution1}
			\begin{split}
			\int_{\mathbb{R}^N}\nabla u_{r_n}\cdot\nabla\varphi dx=&\int _{\mathbb{R}^N}\varphi dx.
			\end{split}
			\end{align}
			Note that we have considered $u_{r_n}>0$ for sufficiently large $n$, since $u_*>0$ in the support of $\varphi$. Passing the limit $n\to\infty$ in \eqref{solution1}, we get 
			\begin{align}\label{solution2}
				\begin{split}
					\int_{\mathbb{R}^N}\nabla u_{*}\cdot\nabla\varphi dx=&\int _{\mathbb{R}^N}\varphi dx.
				\end{split}
			\end{align}
	Since $\varphi\in C_c^{\infty}(\{u_*>0\}\cap\mathcal{C})$ is an arbitrary choice, and for an arbitrary choice of compact set $\mathcal{|C}\subset\mathbb{R}^N$, it establishes that $-\Delta  u_*=1$ in $\{u_*>0\}$. 
	Since $0$ is a free boundary point for $u_*$, we have $u_{r_n}(0)\to 0$ as $n\to\infty$ and $\|u_{r_n}\|_{L^{\infty}(B_r)}\approx r^{2-\frac{N}{q}}$ for $r\in(0,1)$. 
	\end{proof}

		\section*{Appendix}
		We now develope a {\it Weiss-like} monotonicity formula.
		\begin{theorem}\label{monotonicity_formula}
		Let $u$ be a solution to \eqref{main_mod}, then 
		\begin{align}\label{mon1}
		W_u(r):=&\frac{1}{r^{N+6-\frac{2N}{q}}}\int_{B_r}2^{-1}|\nabla u|^2dx-\frac{1}{r^{N+2-\frac{N}{q}}}\int_{B_r}fudx-\frac{1}{r^{N+3-\frac{2N}{q}}}\int_{B_r}u^2dS
		\end{align}
		is monotone. In other words $\frac{d}{dr}W_u(r)\geq 0$ for $r\in(0,1)$.
		\end{theorem}
		\begin{proof}
	   Recall that $u_r(x)=\frac{u(rx)}{r^{2-\frac{N}{q}}}$. We note that
	   $\displaystyle W_u(r)=\int_{B_1}2^{-1}(|\nabla u_r|^2-fu_r)-\int_{\partial B_1}u_r^2dS.$
	   Further, we have
	   \begin{align}\label{mon2}		
	   	\begin{split}
	   	\frac{d}{dr}u_r=&\frac{1}{r}\left\{x\cdot\nabla u_r-\left(2-\frac{N}{q}\right)u_r\right\}\\
	   	\nabla\frac{d}{dr}u_r=&\frac{1}{r}\left\{\nabla (x\cdot\nabla u_r)-\left(2-\frac{N}{q}\right)\nabla u_r\right\}=\frac{d}{dr}\nabla u_r.
	   	\end{split}
	   	\end{align}
	   	Therefore,
	   	\begin{align}\label{mon3}		
	   		\begin{split}
	   			\frac{d}{dr}W_u(r)=&\int_{B_1}\left(\nabla u_r\cdot\nabla\frac{d}{dr}u_r-f\frac{d}{dr}u_r\right)-\int_{\partial B_1}2u_r\frac{d}{dr}u_rdS.
	   		\end{split}
	   	\end{align}
	   	Hereafter we follow the calculations of Ros-Oton \cite{XRO1} to obtain the following:
	   	\begin{align}\label{mon4}		
	   		\begin{split}
	   			\int_{B_1}\nabla u_r\cdot\nabla \frac{d}{dr}u_r=&-\int_{B_1}\Delta u_r\frac{d}{dr}u_r dx+\int_{B_1}\frac{\partial}{\partial r}u_r\frac{d}{dr}u_rdS\\
	   			=&\int_{B_1}f\frac{d}{dr}u_r dx+\int_{B_1}\frac{\partial}{\partial r}u_r\frac{d}{dr}u_rdS.
	   		\end{split}
	   	\end{align}
	   	Thus we have
	   	\begin{align}\label{mon5}		
	   		\begin{split}
	   		\frac{d}{dr}W_u(r)=&\left(-\int_{B_1}\Delta u_r\frac{d}{dr}u_r dx+\int_{B_1}\frac{\partial}{\partial r}u_r\frac{d}{dr}u_rdS\right)-\int_{B_1}f\frac{d}{dr}u_rdx-\int_{\partial B_1}2u_r\frac{d}{dr}u_rdS\\
	   			=&\int_{\partial B_1}\frac{\partial}{\partial r}u_r\frac{d}{dr}u_rdS-\int_{\partial B_1}2u_r\frac{d}{dr}u_rdS.
	   		\end{split}
	   	\end{align}
	   	Therefore, we have 
	   	\begin{align}\label{mon6}		
	   		\begin{split}
	   			\frac{d}{dr}W_u(r)=&\int_{\partial B_1}\frac{1}{r}(x\cdot\nabla u_r-2u_r)\left(x\cdot\nabla u_r-\left(2-\frac{N}{q}\right)u_r\right)dS\\
	   			\geq & \int_{\partial B_1}\frac{1}{r}(x\cdot\nabla u_r-2u_r)^2dS\geq 0.
	   		\end{split}
	   	\end{align}
	   	yielding us the claim made.
		\end{proof}
		
		\begin{lemma}\label{homogeneous}
		Suppose $u$ is a solution to \eqref{main_mod}, then any {\it blow-up} of $u$ at $0$ is homopgeneous of degree $2$.
		\end{lemma}
		\begin{proof}
		By a simple scaling argument of substituting $x'=rx$ where $x\in B_{\rho}$, we obtain that $W_u(\rho r)=W_u(r),$ for any $\rho, r>0$.
		
		Suppose $u_*$ is a {\it blow-up} of $u$ at $0$, then there exists a sequence $r_n\to 0$ as $n\to\infty$ obeying $u_{r_n}\to u_*$ in $C^1_{\text{loc}}(\mathbb{R}^N)$. Hence,
		\begin{align}\label{mon7}
		W_{u_*}(\rho)=\underset{r_n\to 0}\lim W_{r_n}(\rho)=\underset{r_n\to 0}\lim W_u(\rho r_n)=W_u(0).
		\end{align}		
		Monotonicity of  the {\it Weiss-like} function $W$ allows the existence of $\underset{r_n\to 0}\lim W_u(r)=W_{u_*}(0)$.  Thus $W_{u_*}$ is constant in $\rho$ and hence by Theorem \ref{monotonicity_formula} we have $x\cdot\nabla u_{*}-2u_{*}=0.$ 
		A simple application of the Lagrange's method indicates that $u_{*}$ is homogeneous of degree $2$. 
		\end{proof}
		
		\section*{Acknowledgement}
	The author DC has been funded by the National Board for Higher Manematics (N.B.H. M.), D.A.E. (Department of Atomic Energy) for the financial support through the grant (02011/47/2021\-/NBHM(R.P.)/ R\&D II/2615). This project is also partially supported by the Natural Science Foundation of Guangxi Grant Nos. 2021GX\-NSFFA196004 and GKAD23026237, the NNSF of China Grant No. 12371312, the European Union's Horizon 2020 Research and Innovation Programme under the Marie Sklodowska-Curie grant agreement No. 823731 CONMECH, and the project cooperation between Guangxi Normal University and Yulin Normal University.

		\section*{Data Availability Statement}
		
		The article has used no data and hence there is nothing that needs to be shared.
		\section*{Funding and/or Conflicts of interests/Competing interests}
The author declares that there are no conflicts of interest of any type whatsoever.

\end{document}